\documentclass[12pt,a4paper]{article}
\usepackage{amssymb}

\usepackage{graphicx}
\usepackage{amsmath}
\usepackage{amssymb,amsfonts,amsmath,geometry}
\usepackage{amsmath,amssymb}%

\usepackage{amsfonts}
\usepackage{amsthm,amscd}
\usepackage[english]{babel}
\usepackage{graphicx}
\usepackage{amsmath}
\usepackage{amssymb}

\usepackage{amstext}
\usepackage{color}

\newtheorem{theorem}{Theorem}
\newtheorem{corollary}[theorem]{Corollary}
\newtheorem{definition}[theorem]{Definition}
\newtheorem{example}[theorem]{Example}
\newtheorem{lemma}[theorem]{Lemma}

\newtheorem{proposition}[theorem]{Proposition}
\newtheorem{remark}[theorem]{Remark}

\title{Duality results for Iterated Function Systems with a general family of branches}

\author{Jairo K. Mengue,\,\,\, Elismar R. Oliveira }

\begin{document}

\maketitle

\centerline{Instituto de Matem\'atica, UFRGS - Porto Alegre, Brasil}

%\boldmath{   %retirar depois...

%\tableofcontents %retirar depois...

\begin{abstract}
Given $X$, $Y$, $Z$ and $W,$ compact metric spaces we consider two iterated function systems $\{\tau_x: Z\to Z,\, x\in X\}$ and $\{\tau_y:W\to W,\, y\in Y\}$, where $\tau_x$ and $\tau_y$ are contractions. Let $\Pi(\cdot,\cdot,\tau)$ be the set of probabilities $\pi \in \mathcal{P}(X \times Y \times Z\times W)$  with $(X,Z)-$marginal being holonomic with respect to $\tau_x$ and $(Y,W)$-marginal being holonomic  with respect to $\tau_y$. Given $\mu \in \mathcal{P}(X)$ and $\nu \in \mathcal{P}(Y)$, let $\Pi(\mu,\nu,\tau)$ be the set of probabilities in $\Pi(\cdot,\cdot,\tau)$  having $X$-marginal $\mu$ and $Y-$marginal $\nu$. Let $H_\alpha(\pi)$ be  the relative entropy of $\pi$ with respect to $\alpha$ and $H_\beta(\pi)$ be  the relative entropy of $\pi$ with respect to $\beta$. Given a cost function $c\in C(X\times Y \times Z \times W)$, let $P_{\alpha,\beta}(c) = \sup_{\pi\in\Pi(\cdot,\cdot,\tau)} \int c\, d\pi + H_{\alpha}(\pi) +H_{\beta}(\pi)$. We will prove the duality equation:
\[\inf_{_{_{P_{_{\alpha,\beta}}(c -\varphi(x) -\psi(y))=0}}} \int \varphi(x)\,d\mu +\int \psi(y)\,d\nu = \sup_{_{_{\pi\in\Pi(\mu,\nu,\tau)}}} \int c\, d\pi + H_{\alpha}(\pi) +H_{\beta}(\pi).\]
In particular, if $Z$ and $W$ are single points and we drop the entropy, the equation above can be rewritten as the Kantorovich Duality for the compact spaces $X,Y$ and a continuous cost function $-c$.

\end{abstract}

\thispagestyle{empty}

\section{Introduction}

Let $X$ and $Z$ be compact metric spaces. For each $x\in X$ we associate a continuous map $\tau_x: Z \to Z$. We denote by $\Pi(\tau)$ the set of holonomic probabilities in $\mathcal{P}(X\times Z)$ that is, the probabilities $\pi$ satisfying the equation
\[\int g(\tau_x(z)) \, d\pi(x,z) = \int g(z)\, d\pi(x,z), \,\,\, \text{for any}\,\, g\in C(Z).\]

The holonomic constraint above appears in several contexts, not just in discrete dynamics and thermodynamic formalism. For instance, for Lagrangian mechanics, in the Aubry-Mather theory there is an analogous holonomic constraint $\int_{TM} v d\varphi  \; d\mu(x,v)=0, \,\forall \, \varphi \in C^1(M).$ In \cite{BG}, \cite{GO} the Mather measures are the measures that minimizes the integral
$\int_{TM} L(x,v) \; d\mu(x,v)$ among all the probabilities satisfying the holonomic constraint (usually called holonomic probabilities). The discrete version of the Lagrangian holonomic constraint is
$$\int_{TM} \varphi(x+v)  \; d\mu(x,v)=\int_{TM} \varphi(x)  \; d\mu(x,v),$$
what can be seen as the holonomic constraint for an IFS where $\tau_v x= x+v$.

Also, in \cite{PBern}, Section 6, a holonomic constraint is used to study the Monge transportation problem when the cost is the action associated to a Lagrangian function on a compact manifold.
See  \cite{BJ}, \cite{BarnDemElton} and \cite{Urba}, for general properties and some background on iterated function systems theory.

In the rest of this work, we assume that  the  maps $\{\tau_x\}_{x\in X}$ are uniform  contractions, that is, there exists some constant $0<\gamma<1$ such that
\begin{equation}\label{gama}
d(\tau_{x_1}(z_1),\tau_{x_2}(z_2))\leq \gamma [d(x_1,x_2)+d(z_1,z_2)],
\end{equation}
for any $x_i\in X$ and $z_i \in Z$.

For each point $x_0 \in X$ there is a unique fixed point $z(x_0) \in Z$ for $\tau_{x_0}$. Thus, the probability $\pi =\delta_{x_0,z(x_0)}$ is holonomic because
\[ \int g\, \delta_{x_0,z(x_0)} = g(z(x_0)) = g(\tau_{x_0}(z(x_0))) = \int g(\tau_x(z))\, \delta_{x_0,z(x_0)},\]
for $g\in C(Z)$.

The next constructions are inspired by \cite{BCLMS} and \cite{LMMS}. For a fixed $\alpha$ in $\mathcal{P}(X)$, with $\rm{supp}(\alpha)=X$ and a Lipschitz cost function $c(x,z)$  we define the operator $L_{c,\alpha} : C(Z) \to C(Z)$ by
\[L_{c,\alpha}(\psi) (z)= \int e^{c(x,z)}\psi(\tau_x(z))\, d\alpha(x).\]
Sometimes we write $L_{c,\alpha}$ as $L_c$.
With the assumption of $\tilde{\alpha}$ to be a finite measure instead of probability there are not relevant advances in our results. Indeed, when $d\tilde{\alpha}=kd\alpha,$ where $k$ is a positive constant, we have $L_{c,\tilde{\alpha}}=L_{c+\log(k),\alpha}$. For a countable or finite  set $X$, the usual measure $\tilde{\alpha}(x)$ represents the summation over the branches of the weighted IFS (see \cite{BJ}, \cite{OlivLop}, \cite{Urba}). Writing $\tilde{\alpha}(x)=\sum_{i} \delta_{x_{i}}$ we get the transference operator
\[L_{c,\tilde{\alpha}}(\psi) (z)= \sum_{i} e^{c(x_{i},z)}\psi(\tau_{x_{i}}(z)).\]

If  $X=\{1,2,3,4\}$ and $\alpha = (p_1,p_2,p_3,p_4)$ is a positive probability vector, then $L_{c,\alpha}h(z)$ became
$$\sum_{i=1}^{4} e^{c(x_i,z) + \log(p_i)}h(\tau_{x_i}(z)). $$ In this case, the choice of a different probability $\alpha$ only produces a perturbation on the cost function $c$.
When $X$ is not countable, the introduction of the probability $\alpha$ in the operator is natural because it allows us to get the existence of probabilities $\bar{\alpha}$ and $\alpha$ that are not equivalent. In this case, the choice of $\alpha$ is not absorbed by changes on $c$.

The Ruelle-Perron-Frobenius operators, has been studied by several authors in the last few years under different hypothesis on the weights $e^{c(x_{i},z)}$ and on the IFS. For example in \cite{STE} is proved, under suitable hypothesis, the uniqueness of the invariant measures for the operator $T(f)(x)=\int_{S} f(w_s(x))d\mu(s)$ where $\{(X,d), w_s :X\to X,  s \in S\}$ is an IFS  and $\mu$ is an ``\emph{a priori}" distribution of $s$.

Also, in \cite{Urba}, is studied a hyperbolic IFS  with countable many branches, $\{\phi_i :X \to X,  i \in I\}$. In this work, the thermodynamic formalism for the operator
$$\mathcal{L}(\psi)(x)=\sum_{i \in I} e^{\phi^{(i)}} \psi(\phi_{i}(x))$$ was considered, where $\{\phi^{i}\}$ is a family of H{\"o}lder maps of order $\beta$.  In this setting, the topological pressure $P(\phi)=\displaystyle \lim_{n \to \infty} \frac{1}{n} Z_{n}(\phi)$ is defined, where $Z_{n}(\phi)$ is a subadditive partition function associated. Finally, the author shows that the dual operator $\mathcal{L}^{*}$ acting in probabilities has an eigenmeasure with eigenvalue $\lambda= e^{P(\phi)}$ and, under the normalization of $\mathcal{L}^{*}$, there is an unique Gibbs state.

We will briefly show that some results about  XY model \cite{BCLMS} and Spin Lattice Systems \cite{LMMS} can be adapted to the present setting.  In the Chapter 2, we follow \cite{LMMS}  studying the operator $L_c$ and the dual operator $L_c^{*}$. These operators can be used in order to construct the definition of the relative entropy of a holonomic probability $\pi\in\Pi(\tau)$ with respect to $\alpha$ by
\[H_\alpha(\pi) = -\sup\left\{ \int c(x,z)\, d\pi\,\,|\,\, c \,\, \text{is Lipschitz and}\,\, L_c(1)=1\right\}. \]
Thus, the pressure of a continuous function $c$, relative to $\alpha$, is defined by
\[P_\alpha(c) = \sup_{\pi\in \Pi(\tau)}\left[ \int c\, d\pi +H_\alpha(\pi)\right].\]
The number $e^{P(c)}$ is equal to the spectral radius of $L_c$ if $c$ is Lipschitz (related results can be founded in \cite{OlivLop}, \cite{LMMS}). This way of define entropy and pressure is related with the Legendre's transform.

The entropy above defined is a natural generalization of the Kolmogorov-Sinai entropy for symbolic dynamics (see section 2). On the other hand we will present an example where this entropy is not an affine function on the space of holonomic probabilities. Moreover, studying this example, we can easily show that for any $c$, the holonomic measure that attain the supremum in $P(c)$ is not an extremal point of the convex set $\Pi(\tau)$.

From a transport theory point of view (see \cite{Vi1}), is natural try to impose some condition on the $X$-marginal of $\pi$.
For a fixed $\mu\in \mathcal{P}(X)$ we denote by $\Pi(\mu,\tau)$ the set of holonomic probabilities that also satisfy
\[\int f(x)\, d\pi(x,z) = \int f(x)\, d\mu\]
for any $f\in C(X)$, that is the set of holonomic probabilities with $X$-marginal equal to $\mu$.

In the chapter 3, we consider the problem of the marginals. In the general case, given $c\in C(X,Z)$, the possible measures that attain the supremum
\[\sup_{\pi\in \Pi(\tau)} \int c\, d\pi +H_\alpha(\pi)\]
may not be in $\Pi(\mu,\tau)$. So we can ask, what could be said about
\[\sup_{\pi\in \Pi(\mu,\tau)} \int c\, d\pi +H_\alpha(\pi).\]

In order to visualize the Kantorovich duality in the current setting we generalize this question fixing two other compact metric spaces $Y$ and $W$ and an IFS $\{\tau_y:W\to W, \, y\in Y\}$ formed by uniform contractions as $\{\tau_x\}$ above. We denote by $\Pi(\cdot,\cdot,\tau)$ the set of probabilities $\pi \in \mathcal{P}(X\times Y \times Z \times W)$ satisfying $\int f(\tau_x(z))\,d\pi = \int f(z)\, d\pi,\, f\in C(Z)$ and $\int g(\tau_y(w))\,d\pi=\int g(w)\,d\pi,\, g\in C(W).$  This is the set of probabilities with $(X,Z)$-marginal holonomic with respect to $\{\tau_x\}$ and $(Y,W)$-marginal holonomic with respect to $\{\tau_y\}$.
Given probabilities $\alpha \in \mathcal{P}(X)$ and $\beta \in \mathcal{P}(Y)$, with $\rm{supp}(\alpha)=X,\, \rm{supp}(\beta)=Y$, for $\pi \in \Pi(\cdot,\cdot,\tau)$
we denote by $H_{\alpha}(\pi)$ the relative entropy of the holonomic $(X,Z)$-marginal of $\pi$ with respect to $\alpha$. Analogously we denote by $H_{\beta}(\pi)$ the relative entropy of the $(Y,W)$-marginal of $\pi$  with respect to $\beta$. The marginal pressure of a continuous cost function $c(x,y,z,w)$ relative to $(\alpha,\beta)$ will be defined by
\[P_{(\alpha,\beta)}(c) = \sup_{\pi \in \Pi(\cdot,\cdot,\tau)}\int c\, d\pi +H_{\alpha}(\pi) + H_{\beta}(\pi).\]
Using ideas of transport theory (see \cite{Vi1}, \cite{LM}, \cite{LMMS2}), we fix two probabilities $\mu \in \mathcal{P}(X)$ and $\nu \in \mathcal{P}(Y)$ and denote by $\Pi(\mu,\nu,\tau)$ the set of probabilities $\pi\in \Pi(\cdot,\cdot,\tau)$ satisfying
$\int f(x)\, d\pi = \int f(x)\, d\mu(x), \,\,\, f\in C(X) $ and $
\int g(y)\, d\pi = \int g(y)\, d\nu(y), \,\,\, g\in C(Y).$
Applying the Fenchel-Rockafellar Duality Theorem we will prove in Theorem \ref{varprindual}  that
\[\inf_{ _{ _{P_{\alpha,\beta} (c -\varphi(x) -\psi(y))=0}}} \int \varphi(x)\,d\mu +\int \psi(y)\,d\nu = \sup_{ _{ _{\pi\in\Pi(\mu,\nu,\tau)}}} \int c\, d\pi + H_{\alpha}(\pi) +H_{\beta}(\pi),\]
where the possible functions $\varphi(x)$ and  $\psi(y)$ appearing in the left hand side are continuous.

When $Z$ and $W$ are single points and the entropy is dropped (which is the zero temperature case, in the spin lattice system)   $P_{\alpha,\beta}(c-\varphi - \psi)=0$ is equivalent to $\sup_{x,y} \{c(x,y)-\varphi(x) -\psi(y)\}=0$. In this case, the duality above  can be rewritten as the Kantorovich Duality for compact spaces $X$ and $Y$ and a continuous cost function $-c$:
 \[\inf_{c -\varphi(x) -\psi(y)\leq 0} \int \varphi(x)\,d\mu +\int \psi(y)\,d\nu = \sup_{\pi\in\Pi(\mu,\nu)} \int c\, d\pi .\]
 We present this result in Theorem \ref{kantorovich}.

On the other hand, if we suppose that $Y$ and $W$ have a single point, then the Theorem \ref{varprindual} can be rewritten as
\[\inf_{P(c-\varphi(x))=0}\int \varphi d\mu  = \sup_{\pi\in\Pi(\mu,\tau)} \int c\, d\pi + H_\alpha(\pi).\]
This result  can be interpreted as a kind of variational principle using the operator $L_{c,\alpha}$ where we change the concept of eigenvalue from number to a function on the variable $x$.

Although this particular result seems to recall the main result in \cite{LMMS2}, we have a different situation. In order to illustrate the differences, we present an example of application in Thermodynamic Formalism. Let $X=\{0,1\}$ and $Z=\{0,1\}^{\mathbb{N}}$ be compact metric spaces with the respective natural distances. Let $\tau_x(z_1,z_2,...) =(x,z_1,z_2,...)$ and for a fixed probability measure $\mu=(p_1,p_2)$  over $\{0,1\}$, consider the set of invariant probabilities $\pi$ for the shift map $\sigma$ acting on $Z=\{0,1\}^{\mathbb{N}}$ satisfying $\pi([0])=p_0$ and $\pi([1]) = p_1$, where $[0]$ and $[1]$ are the cylinder sets of size one. Denote this set by $\Pi(\mu)$ and consider the variational problem
\[ \sup_{\pi\in\Pi(\mu)} \int A\, d\pi + h(\pi)\]
where $A$ is a Lipschitz function and $h$ is the Kolmogorov entropy. In this case we are interested only on invariant probabilities $\pi$ such that $\pi([0])=p_0$. This supremum is equal to
\[\inf \{\varphi_0 p_0 + \varphi_1 p_1\}\]
where $( \varphi_0,\varphi_1)$ satisfies  $\sum_{i=1,2}e^{A(iz)-\varphi_i}g(iz) = g(z)$ for some positive and continuous function $g$ and any $z$.

\section{Entropy and pressure}
In this section we study the operator $L_c=L_{c,\alpha}$ and the dual operator $L_c^{*}$. We use this operator to define the entropy of a holonomic probability (see \cite{OlivLop}, \cite{LMMS} for related results). 

\begin{lemma}\label{exist} There exists a positive eigenfunction $h(z)$ associated to a positive eigenvalue $\lambda$ for $L_c$.
\end{lemma}
\begin{proof}
This proof follows similar arguments to those found in \cite{bousch} (see also \cite{BCLMS}).
\end{proof}

Using that $\lambda$ and $h$ given above are positive, the function $$\overline{c}(x,z) = c(x,z)+\log(h(\tau_x(z)))-\log(h(z))-\log(\lambda)$$ is well defined. Clearly
$L_{\overline{c},\alpha}(1) = 1$ and we notice that $\overline{c}$ is a Lipschitz function. Indeed, given points $(x_1,z_1)$ and $(x_2,z_2)$, from (\ref{gama}) we have
\[|h(\tau_{x_1}(z_1))-h(\tau_{x_2}(z_2))| \leq C  d(\tau_{x_1}(z_1), \tau_{x_2}(z_2)) \leq C \gamma (d(x_1,x_2) +d(z_1,z_2)).\]
The conclusion follows from the fact  that $c$ and $h$ are Lipschitz functions, and $\log(\cdot)$ is an analytic function.

\begin{definition} A Lipschitz function $c=c(x,z)$ is normalized if $L_c(1)=1$.
\end{definition}

When we write, ``$c$ is normalized", or consider the operator $L_c$, we always suppose that $c$ is a Lipschitz function. Also, the function $\overline{c}$ is always a normalized function associated to $c$. The next results can be used in order to show that $\overline{c}$ is unique. Naturally $\int \overline{c}\, d\pi = \int c\, d\pi -\log(\lambda)$ for any holonomic probability $\pi$.

\bigskip

\begin{example} The present setting do not exclude the possibility of existence of more than one eigenfunction $h\geq 0$ with positive eigenvalue associated.

Consider $X=\{1\}$, $Z=[0,1]$, $\tau_x(z) =z/2 $ and $c=0$.
Then $\tau_x$ is a contraction map with $\gamma=1/2$.
\[L_{c,\alpha}(\psi)(z)=L(\psi)(z)= \psi(z/2).\]
From this, a continuous family of eigenfunctions $h_{\alpha}\geq 0$ associated to positive eigenvalues is
$h_\alpha(z)=z^\alpha, \alpha \neq 0  \,\,\,\,\text{and}\,\,\,\, h_0(z)= 1.$
Indeed
$L(h_\alpha)(z)= (z/2)^\alpha = \frac{1}{2^\alpha}h_\alpha(z), \,\,\, \text{and}\,\,\, L(h_0)(z)= 1= 1 \cdot h_0(z).$ If $\alpha \neq 0$, $h_\alpha(0)=0^\alpha=0$.
\end{example}

\bigskip

In the standard Thermodynamic Formalism one consider cases that are topologically mixing. Hence, if some eigenfunction $h\geq 0$ satisfies $h(x_0)=0$ at some point $x_0$ is easy to show that $h=0$. This result can be used in order to prove that the main eigenvalue $\lambda$ is simple (see \cite{PP}). The argument below shows that we can not exclude cases as the previous one.

\begin{lemma}
There exists a unique positive eigenvalue $\lambda$ associated to a strictly positive continuous eigenfunction for $L_{c,\alpha}$ and this $\lambda$ is equal to the spectral radius of $ L_{c,\alpha}$ over $C(Z)$. The eigenfunction $h$ associated to $\lambda$ is Lipschitz and unique except by multiplication by a  constant.
\end{lemma}
\begin{proof}
From the Lemma~\ref{exist} there exists $\tilde{\lambda}> 0$ and  $\tilde{h}> 0$ such that $L_{c,\alpha}(\tilde h) = \tilde{\lambda} \,\tilde{h}$. The function $\overline{c}(x,z) = c(x,z)+\log(\tilde{h}(\tau_x(z)))-\log(\tilde{h}(z))-\log(\tilde{\lambda})$ is normalized and for any pair $\lambda_0$ and $h_0$ we have that
$L_{\overline{c},\alpha}h_0 = \lambda_0h_0$ iff $L_{c,\alpha}(h_0\cdot\tilde{h}) = (\lambda_0\cdot \tilde{\lambda})(h_0\cdot\tilde{h}).$
Therefore it is sufficient to prove our claim for  $\overline{c}$.

Hence, we need first to show that $\lambda = 1$ is the unique positive eigenvalue that can be associated to a positive eigenfunction. In order to do that we suppose that $L_{\overline{c},\alpha}$ has a positive eigenvalue $\lambda_2$ associated to a strictly positive continuous eigenfunction $h_2$. We want to show that $\lambda_2=1$. Let $h_2(z_0)=\min\{h_2(z)\}$ and $h_2(z_1)=\max\{h_2(z)\}$. Then
\[\lambda_2 h_2(z_1) = \int e^{\overline{c}(x,z)}h_2(\tau_x(z_1))\,d\alpha \leq \int e^{\overline{c}(x,z)} h_2(z_1)\,d\alpha= h_2(z_1)\]
(this shows that $\lambda_2\leq 1$)
and
\[\lambda_2 h_2(z_0) = \int e^{\overline{c}(x,z)} h_2(\tau_x(z_0))\,d\alpha \geq \int e^{\overline{c}(x,z)} h_2(z_0)\,d\alpha= h_2(z_0)\]
(this shows that $\lambda_2 \geq 1$).

In order to show that $1$ is equal to the spectral radius of $L_{\overline{c}}$ we notice that if $|u|_{\infty}=\sup_z\{|u(z)|\} \leq 1$ then
$|L_{\overline{c}}(u)|_{\infty} \leq \int e^{\overline{c}(x,z)}|u|_{\infty}d\alpha(x) = |u|_{\infty}.$
Thus
$ |L^{n}_{\overline{c}}(u)|_{\infty}^{1/n} \leq 1.$

Now we will show that the constant function is the unique eigenfunction for the normalized operator associated to the eigenvalue 1. In order to show this we suppose that there exists an eigenfunction $h$ non constant. We can suppose $h\geq 0$ (because $h+ cte$ will be also an eigenfunction). Let $h(z_0)=\min\{h\}$, $h(z_1)=\max\{h\}$ and $0<\epsilon<h(z_1)-h(z_0)$.
Writing $L_{\overline{c},\alpha}^{n}=L_{\overline{c},\alpha}\circ L_{\overline{c},\alpha}^{n-1}, \text{ for } n=2,3,...$ we have
\begin{align*}
  h(z_1) &=L_{\overline{c},\alpha}^n(h)(z_1) \\
   &  =  \int_{X^{n}}  e^{^{^{\overline{c}(x_n, \tau_{x_{n-1}}...\tau_{x_1}z_1)}}} \cdot \cdot \cdot e^{^{^{\overline{c}(x_2,\tau_{x_1}z_1)}}}e^{^{^{\overline{c}(x_1,z_1)}}}
   h(\tau_{x_n}...\tau_{x_1}z_1) \, d\alpha^{n}(x_1,...,x_n).
\end{align*}
This show that $\max\{h\}=h(\tau_{x_n}...\tau_{x_1}z_1)$ for $\alpha^{n}$-a.e. $(x_1,...,x_n)$. From our hypothesis on $\tau$ (see equation \eqref{gama}), we have
$$d(\tau_{x_n}...\tau_{x_1}z_1,\tau_{x_n}...\tau_{x_1}z_0)<\gamma^{n}diam(Z).$$
The function $h$ is uniformly continuous (because $Z$ is compact), then for a sufficiently large $n$ we have $ |h(\tau_{x_n}...\tau_{x_1}z_1)-h(\tau_{x_n}...\tau_{x_1}z_0)| < \varepsilon.$  Consequently
$$h(\tau_{x_n}...\tau_{x_1}z_0)> \max\{h\}-\epsilon> \min\{h\}=h(z_0),\,\,\, \alpha^{n}\text{-a.e.} \,(x_1,...,x_n).$$ This is a contradiction because from
\[h(z_0) =  \int_{X^{n}} e^{c(x_n, \tau_{x_{n-1}}...\tau_{x_1}z_0)}...e^{c(x_2,\tau_{x_1}z_0)}e^{c(x_1,z_0)}h(\tau_{x_n}...\tau_{x_1}z_0)\, d\alpha^{n}(x_1,...,x_n),\]
we conclude that $h(\tau_{x_n}...\tau_{x_1}z_0)=h(z_0)$, $\alpha^{n}$-a.e. $(x_1,...,x_n)$.

\end{proof}

From the above lemma there is only one way of associate a normalized
function. $\overline{c}$ to a Lipschitz function $c(x,z)$ by adding  a constant and a continuous function (that is also Lipschitz) in the form $g(\tau_x(z))-g(z)$.

\bigskip

If $L_{c,\alpha}(1)=1$, the dual operator $L_{c,\alpha}^{*}$ (denoted also by $L^{*}$) acting over probabilities in  $\mathcal{P}(Z)$ is defined by
\[\int \psi(z)\, dL^{*}(P) = \int L(\psi) \, dP.\]
Let $\rho$ be an invariant probability for the dual operator. We want to prove the uniqueness of $\rho$. In the next lemma $|u|_{\infty}$ is the supremum norm of a Lipschitz function $u$ and $|u|_{lip}$ is the Lipschitz constant of $u$.

\begin{lemma} Suppose that $c$ is normalized and $u$ is a Lipschitz function. There is a constant $C$ that does not depend on $u$ satisfying
\[|L_c^{n}u|_{lip} \leq C|u|_{\infty} +\gamma^{n}|u|_{lip}.\]
\end{lemma}

\begin{proof}
The proof follows from an argument similar to the proof of prop. 2.1  in \cite{PP}.
\end{proof}

For each sequence $x_1,x_2...,$ of elements in $X$ and any points $z_1$, $z_2$ in $Z$ we have
\[d(\tau_{x_n}...\tau_{x_1}z_1,\tau_{x_n}...\tau_{x_1}z_2) < \gamma^{n}diam(Z).\]
Let $\hat Z$ be the set of points $z\in Z$ satisfying that there exists elements $x_1,x_2,...$ in $X$ and $z_0\in Z$ such that $z$ is an accumulation point of $\{\tau_{x_n}...\tau_{x_1}z_0\}_{n\geq 1}$. From the above computation the point $z_0$ is not relevant, but $x_1,x_2,...$.
If for each $k$ there exists a sequence $x_1^{k},...,x_k^{k}$ and a point $z_k\in Z$ such that $d(z, \tau_{x_k^{k}}...\tau_{x_1^{k}}z_k)< 1/k$ then $z\in \hat Z$ because is an accumulation point for the sequence $\{x_1^{1},x_{1}^{2},x_2^{2},x_1^3,x_2^3,x_3^3,... \}$.

\begin{proposition} There is a unique probability $\rho\in \mathcal{P}(Z)$ invariant under the action of the dual operator $L_{c}^{*}$. The support of $\rho$ is a subset of $ \hat Z$ and for any Lipschitz function $u:Z\to\mathbb{R}$ we have
\[L_{c,\alpha}^{n}(u) \to \int u\, d\rho\]
uniformly in $Z$.
\end{proposition}

\begin{proof}
We will show first that, for any possible fixed probability $\rho$,  $\rm{supp}(\rho) \subseteq \hat Z$.
Consider a point $z_1 \in Z-\hat Z$ then there exists $\epsilon>0$ and $n_0$ such that, for any $z_0\in Z$ and $x_1,...,x_n$ in $X$ with $n\geq n_0$
$d(z_1, \tau_{x_n}...\tau_{x_1}z_0) > \epsilon.$
Let $u:Z\to [0,+\infty)$ be a continuous function satisfying  $u=1$ in $B(z_1, \epsilon/2)$ and $u=0$ out of $B(z_1,\epsilon)$. Then, for $n\geq n_0$ we have the equality
$ \int u\, d\rho = \int L_{c}^{n}u \, d\rho=$
\[ = \int_Z \int_{X^{n}} e^{c(x_n, \tau_{x_{n-1}}...\tau_{x_1}z)}...e^{c(x_2,\tau_{x_1}z)}e^{c(x_1,z)}u(\tau_{x_n}...\tau_{x_1}z)\, d\alpha^{n}(x_1,...,x_n)d\rho(z)=0.\]
Therefore $z_1 \notin \rm{supp}(\rho)$.

Now we will proof the other claims in the proposition.
Consider a fixed Lipschitz function $u:Z\to [0,\infty)$. From the above lemma
$|L_c^n(u)|_{lip} \leq C|u|_{\infty} +\gamma^{n}|u|_{lip}$
and we also have
$|L_c^n(u)|_{\infty} \leq |u|_{\infty}.$
Thus $\{L_c^{n}u\}$ is an equicontinuous family. From the Arzela-Ascolli theorem there is a uniformly convergent sub-sequence $L_c^{n_i}u \to w$.
As $u\geq 0$, we have $w\geq 0$. From the inequalities $\sup\{u\} \geq \sup\{L_c(u)\} \geq \sup\{L_c^{2}(u)\}...$  and $\inf\{u\} \leq \inf\{L_c(u)\} \leq \inf\{L_c^{2}(u)\}...$ we conclude that $\sup\{w\}= \sup\{L_c^{n}(w)\},$ for $\, n\geq 1$ and $\inf\{w\}= \inf\{L_c^{n}(w)\},$ for $\, n\geq 1$.

We claim that $w$ is a constant function.\newline
Indeed, suppose that
$\sup\{w\} - \inf\{w\} > \epsilon>0$. Let $\{z_n^{1}\}$ and $\{z_n^{2}\}$ be such that $\sup\{w\} = (L_c^{n}w)(z_n^1)$ and $\inf\{w\}=(L_c^{n}w)(z_n^2)$, then
\[
\sup\{w\} = \int_{X^{n}}e^{c(x_n,\tau_{x_{n-1}}...\tau_{x_1}z_n^1)+...+c(x_1,z_n^1)}w(\tau_{x_n}...\tau_{x_1}z_n^1) \,d\alpha^{n}(x_1,...,x_n).
\]
Therefore $w(\tau_{x_n}...\tau_{x_1}z_n^1)=\sup\{w\}$, $\alpha^{n}$-a.e. $(x_1,...,x_n)$.
From the  hypothesis on $\tau$ (see equation \eqref{gama}) we have
$$d(\tau_{x_n}...\tau_{x_1}z_n^1,\tau_{x_n}...\tau_{x_1}z_n^2)<\gamma^{n}diam(Z).$$
The function $w$ is uniformly continuous, then for $n$ large enough we have  $ |w(\tau_{x_n}...\tau_{x_1}z_n^1)-w(\tau_{x_n}...\tau_{x_1}z_n^2)| < \epsilon.$  Consequently
$$w(\tau_{x_n}...\tau_{x_1}z_n^2)> \sup\{w\}-\epsilon> \inf\{w\},\,\,\, \alpha^{n}\text{-a.e.} \,(x_1,...,x_n).$$ This is a contradiction because, from
\[\inf\{w\} =   \int_{X^{n}}e^{c(x_n,\tau_{x_{n-1}}...\tau_{x_1}z_n^2)+...+c(x_1,z_n^2)}w(\tau_{x_n}...\tau_{x_1}z_n^2) \,d\alpha^{n}(x_1,...,x_n)\]
we can also conclude that $w(\tau_{x_n}...\tau_{x_1}z_n^2)=\inf\{w\}$, $\alpha^{n}$-a.e. $(x_1,...,x_n)$. Thus $w$ is constant.

Let $\rho$ be an invariant probability for $L_c^*$. Then $\rm{supp}(\rho) \subseteq \hat Z$ and
\[\int u \, d\rho = \lim_{i\to\infty} \int L_c^{n_i}u\, d\rho = \int w\, d\rho = w.\]
Using this equation and the above statement we can conclude that these arguments can be applied for any sub-sequence proving that the sequence $L_c^{n}u $ converges uniformly to $\int\,u\,d\rho$. Considering any possible Lipschitz function $u$ we obtain that $\rho$ is unique and satisfy
\[\int u\, d\rho = \lim_{n\to\infty}(L_c^{n}u),\]
for any  Lipschitz  function $u$.

\end{proof}

The probability $\pi \in \mathcal{P}(X\times Z)$ defined by
\begin{equation}\label{holonomic}
\int g(x,z)\, d\pi = \int \int e^{c(x,z)}g(x,z)\, d\alpha(x)d\rho(z)
\end{equation}
is holonomic.
Indeed, if $g\in C(Z)$ we have
\begin{align*}
 \int g \, d\pi & =\int g\, d\rho =  \int L(g)\, d\rho \\
    & = \int \int e^{c(x,z)}g(\tau_x(z))\,d\alpha(x) d\rho(z)\\
    &= \int g(\tau_x(z))\, d\pi.
\end{align*}

We denote $\pi$ the holonomic probability associated to the pair $c$ and $\alpha$. If $c$ is not normalized then there is a unique normalized function $\overline{c}$ associated to $c$ (following the above discussion) and the holonomic probability associated to $c$ will be (by convention) the holonomic probability associated to $\overline{c}$ .

Following \cite{LMMS} we define the relative entropy of a holonomic probability $\pi\in \mathcal{P}(X\times Z)$ with respect to the probability $\alpha\in \mathcal{P}(X)$ by
\begin{align*}
H_{\alpha}(\pi) &= -\sup_{L_{c,\alpha}1=1} \int c(x,z)\, d\pi(x,z)\\
&=-\sup\left\{ \int c(x,z)\, d\pi(x,z)\,\, | \,\, c\,\, \text{is Lipschitz and } L_{c,\alpha}1 =1\right\}.
\end{align*}
%(writing $L_c$ or $L_{c,\alpha}$ we suppose $c$ Lipschitz). The right hand side above can be %rewritten  as

The pressure of a {continuous} function $c$ with respect to $\alpha$ is defined by
\[P_{\alpha}(c) = \sup_{\pi \in \Pi(\tau)} \int c\, d\pi + H_{\alpha}(\pi).\]

As a consequence of these definitions, part of our results and proofs below will be improvements of the ones that appear in \cite{LMMS}. The following properties are direct consequences of definitions.

\begin{proposition} Properties of entropy and pressure: \newline
1 - $H_\alpha$ is a concave function on $\Pi(\tau)$;\newline
2 - $H_\alpha$ is upper semi-continuous;\newline
3 - $H_\alpha\leq 0$ (because $c=0$ is normalized);\newline
4 - $P_\alpha$ is a convex function on $C(X,Z)$;\newline
5 - if $k$ is a constant, then $P(c+k) = P(c) +k$; \newline
6 - for any $g\in C(Z)$, $P(c(x,z) +g(z) - g(\tau_x(z))) = P(c)$;\newline
7 - The pressure is a continuous function on $(C(x,z), |\,|_{\infty})$, more precisely
$|P(c_1)-P(c_2)|\leq |c_1-c_2|_\infty$.
\end{proposition}

\vspace{0.5cm}

In these notes, is perfectly possible to change the sign of the relative entropy  defining
\[I_\alpha(\pi) = \sup_{L_{c,\alpha}1=1} \int c(x,z)\, d\pi(x,z)\]
and then
\[P_{\alpha}(c) = \sup_{\pi \in \Pi(\tau)} \int c\, d\pi - I_{\alpha}(\pi).\]
In this case the entropy will be non negative. See \cite{VE} for an interesting discussion about the relative entropy.

\bigskip

\begin{example} Suppose that $Z=\{z\}$ has a single point and $X=\{1,...,d\}$ is a finite set. Consider a fixed probability $\alpha=(p_1,p_2,...,p_d)$ where $p_i>0$ is the mass of the point $i\in X$.
Any probability $q=(q_1,...,q_d), \, q_i>0,$ over $X$ can be identified with a probability $\pi=\pi(q)=q$ over $X\times Z$. We have $\tau_x(z)=z$, $x=1,...,d$. We notice that $\pi=q$ is holonomic. Indeed, if $g$ is a function on the variable $z$ then $g$ is constant and $g(z)=g(\tau_x(z))$.
Any function $c(x,z)$ is identified with a function $c(x)$ and will be normalized if
\[\sum_{i=1}^{d}e^{c(i)}p_i =\sum_{i=1}^{d}e^{c(i)+\log(p_i)}= 1.\]
We know that $c(x)= \log(\frac{q_x}{p_x})$ is normalized because
\[   \sum_{i=1}^{d}e^{\log(\frac{q_i}{p_i})+\log(p_i)} = \sum_{i=1}^{d}e^{\log(q_i)}= \sum_{i=1}^{d}q_i =1.\]
 From Jensen's inequality we have\footnote{ We will prove the first equality in the Corollary \ref{associated} because $q$ is the holonomic probability associated to the normalized function $\log(\frac{q_x}{p_x})$.}
\[I_{\alpha}(q)= \sum_{i=1}^{d} \log(\frac{q_i}{p_i})q_i =\sum_{i=1}^{d} \log(q_i)q_i-  \sum_{i=1}^{d} \log(p_i)q_i\geq 0.\]
This is the  Kullback-Leibler entropy of $q$ relative to $p$. Thus,
\[H_\alpha(q) = -\sum_{i=1}^{d} \log(\frac{q_i}{p_i})q_i =-\sum_{i=1}^{d} \log(q_i)q_i+  \sum_{i=1}^{d} \log(p_i)q_i\leq 0.\]
If $p_i=1/d$ then \[H_\alpha(q) = -\sum_{i=1}^{d} \log(q_i)q_i -\log(d) = h(q) -\log(d)\]
where $h(q)$ is the Shannon entropy of $q$.
\end{example}

\bigskip

\begin{example} Consider $X=\{1,2\}$, $Z=\{1,2\}^{\mathbb{N}}$ and $\tau_x(z_1,z_2,...)=(x,z_1,z_2,...)$. Then the IFS is defined by the inverse branches of the shift map $\sigma: Z \to Z$ given by $\sigma(z_1,z_2,...) = (z_2,z_3,...)$. In this case $X\times Z$ can be identified with $Z$ using the map
\[(x,(z_1,z_2,z_3,...))\to(x,z_1,z_2,z_3,...)\] and the holonomic measures in $X\times Z$ coincide with the $\sigma$-invariant measures in $Z$, if we follow this identification. Given $\alpha = (1/2,1/2)$ we have
\begin{align*}
  L_{c,\alpha}(h)(z_1,z_2,...) & = \sum_{i=1,2} e^{c(i,(z_1,z_2,...))}h(\tau_i(z_1,z_2,...))\frac{1}{2}\\
   & \sum_{i=1,2} e^{c(i,z_1,z_2,...))-\log(2)}h((i,z_1,z_2,...)).
\end{align*}
This is the Ruelle Operator associated to the function $c-\log(2)$ in the Thermodynamic Formalism.
The Kolmogorov-Sinai entropy of an invariant measure $\mu\in P(Z)$ satisfies:
$$h(\mu) = -\sup\left\{ \int A\, d\mu \,\,|\,\, \sum_{i=1,2}e^{A(ix)}=1,\,\, A\,\,\text{Lipschitz} \right\}.$$
Then we have $H_{\alpha}(\mu) = h(\mu) - \log(2)$.  If we take a finite measure $\tilde{\alpha}(x)=\sum_{i} \delta_{x_{i}}$ instead of $\alpha=(1/2,1/2)$, then $H_{\tilde{\alpha}}(\mu) = h(\mu)$.
\end{example}

\noindent
\begin{remark} We will show in corollary \ref{legendre} that
\[H_{\alpha}(\pi) = -\sup_{ c \,\text{continuous}} \left[\int c\, d\pi - P_\alpha(c)\right].\]
An analogous characterization in Dynamical Systems can be found in \cite{P} theorem 9.12.

\vspace{1cm}

Consider, for a Lipschitz function $c$, the operator $\hat L_{c}:C(X\times Z)\to C(Z)$ defined by
$$\hat L_c(g)(z) = \int e^{c(x,z)}g(x,z)\,d\alpha(x).$$
If $\pi$ is the holonomic probability associated to the normalized function $c$ then,
from ($\ref{holonomic}$) we obtain that for any $g\in C(X \times Z )$:
\begin{equation}\label{holonomic2}
\int g(x,z)\, d\pi = \int \int e^{c(x,z)}g(x,z)\, d\alpha(x)d\rho(z)=\int \hat L_c(g) d\rho(z)=\int \hat L_c(g) d\pi(x,z)
\end{equation}
%where we use that
%\[\int \int e^{c(x,z)}g(x,z)\, d\alpha(x)d\rho(z)=\int \int e^{c(x,z)}g(x,z)\, d\alpha(x)d\pi(x,z)\]
because $\hat L_c(g) \in C(z)$ and the $z$-marginal of $\pi$ is $\rho$.
\end{remark}

\begin{lemma} \label{lema entropia}
 Given a normalized function $c_0$ with associated holonomic probability $\pi_0$ we have that for any normalized function $c$,
$ \int c \,d\pi_0 \leq \int c_0 \,d\pi_0. $
\end{lemma}

\noindent
\begin{proof}
Using that $e^{-c+c_0}$ is a positive function, we can write
$1=u(x,z)e^{-c(x,z)+c_0(x,z)},$ where $u=e^{c(x,z)-c_0(x,z)}$ is also positive.
Note that, in this case,
$ 1=\hat L_c1
 =\hat L_{c_0}u.$

Hence,
$$  0=\log \bigg(\frac{\hat L_c 1}{1}\bigg)= \log(\frac{\hat L_{c_0}u}{u(x,z)e^{-c(x,z)+c_0(x,z)}})= \log(\hat L_{c_0}u)-\log u+c-c_0,  $$
therefore, $ 0 = \int \log( \hat L_{c_0}u)\,d\pi_0-\int \log u \,d\pi_0 + \int c \,d\pi_0 - \int c_0 \,d\pi_0.$ From ($\ref{holonomic2}$) and from Jensen's inequality we get
\[\int c_0 \,d\pi_0 = \int \log( \hat L_{c_0}u)\,d\pi_0-\int \log u \,d\pi_0 + \int c \,d\pi_0\]
\[=\int \log( \hat L_{c_0}u)\,d\pi_0-\int \hat L_{c_0}( \log u) \,d\pi_0 + \int c \,d\pi_0\]
\[= \int \log( \int e^{c_0}u\,d\alpha)\,d\pi_0-\int \int e^{c_0} \log u \,d\alpha \,d\pi_0 + \int c \,d\pi_0 \geq \int c \,d\pi_0.\]
\end{proof}

\begin{corollary}\label{associated} If $\pi$ is the holonomic probability associated to the normalized cost function $c$, then $H_\alpha(\pi)=-\int c\,d\pi$.
\end{corollary}

It is known that the Kolmogorov-Sinai entropy in symbolic dynamics is an affine function on the space of invariant probabilities (see \cite{P}, Theorem 8.1). The next example shows that this property is false, in general, for holonomic measures

\vspace{0.5cm}

\begin{example} Let $X=\{1,2\}$ and $Z=\{1,2\}$. Define $\tau_{x}(z)=x,\,\, x=1,2,\, z=1,2$. In order to satisfy (\ref{gama}) suppose that $d(1,2) = 2$ on $X$ and $d(1,2)= 1$ on $Z$.  Any probability $\pi \in X\times Z$ can be identified with a $2\times 2$ matrix
\[\pi = \left( \begin{array}{ll}\pi_{11} & \pi_{12} \\ \pi_{21} & \pi_{22} \end{array}\right),\,\,\, \pi_{ij}\geq 0, \,\,\pi_{11}+\pi_{12}+\pi_{21}+\pi_{22} = 1.\]
 A  probability $\pi$ is holonomic if and only if the matrix is symmetric. Consider the holonomic probabilities
\[\pi = \left( \begin{array}{ll}1/4 & 1/4 \\ 1/4 & 1/4 \end{array}\right),\,\,\,\, \eta^1=\left( \begin{array}{cc}1/2 & 0 \\ 0 & 1/2 \end{array}\right) \,\,\,\,\text{and}\,\,\,\,
\eta^2=\left( \begin{array}{cc}0 & 1/2 \\ 1/2 & 0 \end{array}\right)\]
and $\alpha = (1/2,1/2)$ a fixed vector of probabilities on $X$. We have $\pi = \frac{1}{2}\eta^1 + \frac{1}{2}\eta^2$, but $H_\alpha(\pi)=0$ while $H_\alpha(\eta^1) $ and $H_\alpha(\eta^2)$ are negative numbers.
Indeed, $\pi$ is the holonomic probability associated to the normalized function $c=0$, then $H_{\alpha}(\pi) = 0$. Consider
\[c^1(x,z) = \left\{ \begin{array}{ll} \log(1.5), & \text{if}\,\, x=z \\ \log(0.5), & \text{if}\,\, x\neq z\end{array}\right. \]
this function $c^1$  is normalized, then
\[H_{\alpha}(\eta^{1}) \leq - \int c^1\,d\eta^1 = - \log(1.5).\] In the same way, the function
\[c^2(x,z) = \left\{ \begin{array}{ll} \log(1.5), & \text{if}\,\, x\neq z \\ \log(0.5), & \text{if}\,\, x= z\end{array}\right.\]
is also normalized, then
\[ H_{\alpha}(\eta^{2}) \leq - \int c^2\,d\eta^2 = - \log(1.5).\]
If we consider the entropy relative to $\tilde{\alpha}=\delta_1 + \delta_2$ then it can be shown that
\[H_{\tilde{\alpha}}(\pi) = \log(2)\,\,\text{and}\,\,\, H_{\tilde{\alpha}}(\eta^1)= H_{\tilde{\alpha}}(\eta^2)=0.\]
\end{example}

\bigskip

\begin{theorem}[Variational Principle]\label{varprin} If $c$ is a Lipschitz function, then
$P_{\alpha}(c)$ is equal to $\log(\lambda)$ where $\lambda=\lambda_c$ is the unique positive eigenvalue  associated to a positive eigenfunction $h$ for $L_{c,\alpha}$. If $\pi$ is the holonomic probability associated to the normalized cost function $\overline{c}:= c +\log(h\circ\tau_x) - \log(h) - \log(\lambda_c)$, then
\[P_\alpha(c) = \int c\, d\pi + H_\alpha(\pi).\]
\end{theorem}
\begin{proof}
Let $\lambda_c>0$ be the eigenvalue and $h>0$ the eigenfunction associated to $c$, then $\overline{c}(x,z):= c(x,z) +\log(h\circ\tau_x)(z) - \log(h)(z) - \log(\lambda_c)$ is the normalized cost associated to $c$. As $h$ depends only on $z$, for any $\pi \in \Pi(\tau)$ we have that
$\displaystyle \int \bar c\, d\pi =\int c\, d\pi -\log(\lambda_c).$ From the definition of  entropy, we obtain for any $\pi \in \Pi(\tau)$ that $H_\alpha(\pi)\leq -\int \bar c\, d\pi$. Then

\begin{equation}\label{pressure}
P(c) = \sup_{\pi \in \Pi(\tau)} \left(\int c\, d\pi + H_\alpha(\pi)\right) \leq  \sup_{\pi \in \Pi(\tau)} \left(\int c\, d\pi -\int \bar c\, d\pi \right) = \log(\lambda_c).
\end{equation}

In order  to show the other inequality, let  $\pi_{\bar{c}}$ be the holonomic probability associated to  $\bar c$. Then, from the Corollary~\ref{associated} we get
\[P(c) \geq \int c \, d\pi_{\bar c} + H_\alpha(\pi_{\bar c}) = \int c \, d\pi_{\bar c} - \int \bar c \, d\pi_{\bar c} = \log(\lambda_c).\]
\end{proof}

We will see in the next section, that the above Variational Principle can be interpreted as a duality equation.

\begin{corollary}\label{legendre} For a holonomic probability $\pi$:
\[H_{\alpha}(\pi) = -\sup_{ c \,\text{continuous}} \left[\int c\, d\pi - P_\alpha(c)\right] = -\sup_{P_\alpha(c)=0, \, c \,\text{continuous}} \int c\, d\pi.\]
\end{corollary}

\begin{proof}
We have
\[I_\alpha(\pi) = \sup_{ c \,\text{Lipschitz}} \left[\int c\, d\pi - P_\alpha(c) \right]\leq \sup_{ c \,\text{continuous}} \left[\int c\, d\pi - P_\alpha(c)\right].\]
Using that any continuous function $c$ can be approximated by Lipschitz functions in the uniform topology, and that $P_\alpha$ is continuous we get the equality.
\end{proof}

\bigskip

The pressure is a convex function on $C(X\times Z)$ and, at the standard Thermodynamic Formalism any invariant probability attaining the supremum is ergodic. For the holonomic case, the set of holonomic measures is convex, but the probabilities attaining the supremum can not be extreme points of this convex set.

\bigskip

\noindent
\begin{example}
We will consider the same IFS as in a previous example. Let $X=\{1,2\}$, $Z=\{1,2\}$ and $\tau_{x}(z)=x,\,\, x=1,2,\, z=1,2$. We recall that  a probability $\pi$ is holonomic if and only if the associated matrix is symmetric. Consider the holonomic probabilities
\[\pi^1 = \left( \begin{array}{ll}1 & 0 \\ 0 & 0 \end{array}\right),\,\,\,\, \pi^2=\left( \begin{array}{cc}0 & 1/2 \\ 1/2 & 0 \end{array}\right)\,\,\,\,\text{and}\,\,\,\,
\pi^3=\left( \begin{array}{cc}0 & 0 \\ 0 &1 \end{array}\right)\]
and  $\alpha = (1/2,1/2)$ a fixed vector of probabilities on $X$. We can easyly  see that these probabilities are the extreme points of the convex set formed by holonomic probabilities. One can show that $H_\alpha(\pi^1)=H_\alpha(\pi^2)=H_\alpha(\pi^3)=-\log(2)$. Given a normalized cost function $c$ we obtain
\[\int c\, d\pi^1 + H_\alpha(\pi^1)= \int c\, d\pi^1 -\log(2) =c(1,1) -\log(2)< 0\]
because $e^{c(1,1)}+e^{c(2,1)}=2$. As $c$ is normalized we get
\[\int c\, d\pi^1 + H_\alpha(\pi^1) < 0 = \sup_{\pi \in \Pi(\tau)}\int c\, d\pi + H_\alpha(\pi).\]
Similar computations can be made for $\pi^2$ and $\pi^3$.
\end{example}

\section{Duality results}
Part of this section contains ideas previously developed in \cite{LM} and \cite{LMMS2}.
Given a normed linear space $E$ and a convex function $\Theta:E\to\mathbb{R}\cup \{+\infty\}$, the \textit{Legendre-Fenchel transform } of $\Theta$ is the function $\Theta^{*}:E^{*}\to \mathbb{R}\cup\{+\infty\}$, given by
\begin{equation}
\label{transformada}
\Theta^{*}(f)=\sup\limits_{x\in E} [f(x)-\Theta(x)].
\end{equation}

\begin{theorem}[\textbf{Fenchel-Rockafellar duality, \cite{Vi1}}]\label{FR}
\label{Fenchel}
Suppose that $E$ is  a normed vector space,  $\Theta$ and $\Xi$ are two convex functions defined on $E$ taking values in $\mathbb{R}\cup \{+\infty\}$. Denote $\Theta^{\ast}$ and  $\Xi^{\ast}$, the Legendre-Fenchel transform of  $\Theta$ and $\Xi$, respectively.
Suppose that there exists  $v_0\in E$, such that $\Theta(v_0)<+\infty,\, \Xi(v_0)<+\infty$ and that $\Theta$ is continuous on $v_0$.
Then,
\begin{equation}
\inf_{v \in E}[\Theta(v)+\Xi(v)]=\sup_{f\in E^{*}}[-\Theta^{*}(-f)-\Xi^{*}(f)] \label{rockafeller}
\end{equation}
Moreover, the supremum in ($\ref{rockafeller}$) is attained in at least one element in $E^*$.
\end{theorem}

We recall that we did, in the previous section, for a Lipschitz function $c(x,z)$ the normalization
\[\overline{c}(x,z) = c(x,z) + \log(h(\tau_x z)) - \log(h(z)) - \log(\lambda).\]
This normalization is closely related to the variational principle given in Theorem \ref{varprin}.
Different normalizations  imply  different duality results for the variational principles. In what follows, we show some examples concerning this fact. Actually, we can replace the equality in the above normalization by an inequality.

Now we consider two more compact spaces $Y$ and $W$ and uniform contractions $\tau_y: W \to W$ defining an IFS in the same way that $\tau_x: Z \to Z$. Then the results contained in the previous section can be applied if we replace $X$ by $Y$ and $Z$ by $W$. Now we have the spaces $X,Y,Z$ and $W$, and two IFS $\{\tau_x(z)\}$ and $\{\tau_y(w)\}$.

Denote by $\Pi(\cdot,\cdot,\tau)$ the set of probabilities $\pi \in \mathcal{P}(X\times Y \times Z\times W)$ satisfying
\begin{equation}\label{biholonomic}
\int g(\tau_x(z))\,d\pi = \int g(z)\,d\pi \, \,\,\text{and}\,\,
\int g(\tau_y(w))\, d\pi = \int g(w)\, d\pi,\,
\end{equation}
where $g\in C(Z)$ or $g\in C(W)$ respectively.
We chose fixed probabilities $\alpha(x)$ and $\beta(y)$ satisfying $\rm{supp}(\alpha)=X$ and $\rm{supp}(\beta)=Y$. For $\pi \in \Pi(\cdot,\cdot,\tau)$, denote by $H_\alpha(\pi)$ the relative entropy of the $(X,Z)$-marginal of $\pi$ with respect to $\alpha$. In the same way denote by $H_\beta(\pi)$ the relative entropy of the $(Y,W)$- marginal of $\pi$ with respect to $\beta$.
We define the marginal pressure of a continuous cost function $c(x,y,z,w)$ relative to $(\alpha,\beta)$ by
\[ P_{\alpha,\beta}(c) = \sup_{\pi \in \Pi(\cdot,\cdot,\tau)} \int c(x,y,z,w) \, d\pi +H_{\alpha}(\pi) + H_\beta(\pi).\]

\begin{proposition}\label{varpri2} Given a continuous cost $c=c(x,y,z,w)$, consider  the set  $\Phi$ containing the numbers $\lambda$ such that
\[ c(x,y,z,w)  -\lambda \leq b(x,z) +d(y,w)\]
for some continuous functions $b(x,z)$ and $d(y,w)$ with $P_\alpha(b)=P_\beta(d)=0$. Then
\[P_{\alpha,\beta}(c) = \inf\{\lambda: \lambda \in \Phi\}.\]
\end{proposition}

This proposition will be proved below.

Following ideas of Transport Theory, we chose fixed  probabilities $\mu \in \mathcal{P}(X)$, $\nu \in \mathcal{P}(Y)$ and consider the set of probabilities $\Pi(\mu,\nu,\tau) \subset \Pi(\cdot,\cdot,\tau)$ containing the probabilities $\pi$ that also satisfy
\begin{equation}\label{biholonomic2}
\int f(x)\, d\pi = \int f(x)\,d\mu, \,\,f\in C(X)\,\,\,\text{and}\,\,\,
\int g(y)\, d\pi = \int g(y)\,d\nu, \,\,g\in C(Y).
\end{equation}
This is the set of probabilities $\pi\in \mathcal{P}(X\times Y\times Z \times W)$ with $X$-marginal equal to $\mu$, $Y$-marginal equal to $\nu$, $(X,Z)$-marginal holonomic with respect to $\tau_x$ and $(Y,W)$-marginal holonomic with respect to $\tau_y$.

In order to show that this set is not empty, consider for each $x\in X$ and $y\in Y$ the points $z_x$ and $w_y$ that are the fixed points for the contractions $\tau_x$ and $\tau_y$ respectively.
Let $\pi$ be the probability  defined by $d\pi = (d\delta_{x,z_x}d\mu(x))(d\delta_{y,w_y}d\nu(y))$, what means that
\[ \int g(x,y,z,w) \,d\pi  = \int \int g(x,y,z_x,w_y) \,d\mu(x)d\nu(y).\]
If $g\in C(Z)$ we have
\[\int g(z)\, d\pi = \int g(z_x)\,d\mu(x) = \int g(\tau_x(z_x))\,d\mu(x)  = \int g(\tau_x(z))\,d\pi.\]
If $g\in C(W)$ we have
\[\int g(w)\, d\pi = \int g(w_y)\,d\nu(y) = \int g(\tau_y(w_y))\,d\nu(y)  = \int g(\tau_y(w))\,d\pi.\]
This shows that $\pi$ satisfies (\ref{biholonomic}), and is clear that $\pi$ satisfies (\ref{biholonomic2}).

\begin{theorem}(Duality)\label{varprindual} For a continuous cost $c(x,y,z,w)$ we have
\[\inf_{_{_{P_{\alpha,\beta}(c -\varphi(x) -\psi(y))=0}}} \int \varphi(x) \, d\mu + \int \psi(y)\, d\nu =
\sup_{_{_{\pi\in \Pi(\mu,\nu,\tau)}}} \int c(x,y,z,w)\, d\pi + H_{\alpha}(\pi) + H_{\beta}(\pi).\]
\end{theorem}

\begin{proof} The structure of this proof  is close to \cite{Vi1} and \cite{LMMS2}.  In order to make the computations let $E=C(X\times Y\times Z \times W)$. We can suppose that $c\leq 0$. Indeed, if we add a constant in $c$ then we change the booth sides in the same form.

Define $\Theta,\Xi: E\longrightarrow \mathbb{R}\cup\{+\infty \}$ from
$$\Theta(u)=\left\{\begin{array}{ll}
0,&  \mbox{if} \ u(x,y,z,w)\geq c(x,y,z,w)-  b(x,z)-d(y,w) \, \forall(x,y,z,w)\\ & \mbox{ for some }  b,d\,\, \mbox{continuous with }\, P_{\alpha}(b)=0=P_{\beta}(d) ,\\
\\

+\infty,&  \mbox{otherwise}
\end{array}\right.
$$
and
$$
\Xi(u)=\left\{\begin{array}{ll}
\int_X\varphi \,d\mu + \int_Y \psi\, d\nu,&  \mbox{if} \ u=\varphi(x)+\psi(y)-g(\tau_x(z))+g(z)-f(\tau_y(w))+f(w), \\
&\mbox{where} \ \text{the functions are continuous}\,,\\
\\
+\infty,&  \mbox{otherwise}.
\end{array}\right.
$$

\bigskip

$\Xi$ is well defined because if it is not $+\infty$, then it coincides with $\int u\, d\pi, \pi \in \Pi(\mu,\nu,\tau)$. We recall that $P_\alpha(0)=P_\beta(0) = 0$. Then if $u\geq c$ we have $\Theta(u)=0$. We also recall that any normalized function has zero pressure.

 The hypothesis in Theorem \ref{FR} are satisfied. Indeed, taking $u$ constant sufficiently large $\Theta$ is continuous in $u$, $\Theta(u)<\infty$ and $\Xi(u)<\infty$. Clearly $\Xi$ is a convex function. In order to show that $\Theta$ is convex suppose that $\Theta(u_1)=\Theta(u_2)=0$. We can write  $u_1\geq c- b_1-d_1$, $u_2 \geq c-b_2-d_2$, $P(b_i)=P(d_i)=0$ and
$\lambda u_1+ (1-\lambda)u_2 \geq c  -(\lambda b_1 +(1-\lambda)b_2) -(\lambda d_1 +(1-\lambda)d_2).$
 From the convexity of the pressure we get $P_{\alpha}(\lambda b_1 +(1-\lambda)b_2) \leq 0$, so there exists a constant $a_1\geq 0$ such that $P_{\alpha}(\lambda b_1 +(1-\lambda)b_2+a_1)=0$. In the same way there exists a constant $a_2\geq 0$ such that $P_{\beta}(\lambda d_1 +(1-\lambda)d_2+a_2)=0$, and we have
\[  \lambda u_1+ (1-\lambda)u_2 \geq c  -(\lambda b_1 +(1-\lambda)b_2+a_1) -(\lambda d_1 +(1-\lambda)d_2+a_2),\]
showing that $\Xi(\lambda u_1+ (1-\lambda)u_2)=0$.

For any  $\pi\in E^{*}$, we get
\begin{eqnarray*}
&&\Theta^*(-\pi) =\sup_{u\in E}\left\{\langle -\pi,u \rangle- \Theta(u) \right\}\\
&&= \sup_{u\in E}\left\{\langle\pi,u\rangle : \ u\leq - c + b+d,\,\,  P_\alpha(b)=0=P_\beta(d) \right\}\\
&&=\left\{\begin{array}{ll}
\displaystyle\langle \pi,\, - c \,\rangle +\sup_{b\,: \,  P_\alpha(b)=0}\langle \pi, b\rangle + \sup_{d\,: \,  P_\beta(d)=0}\langle \pi, d\rangle\,,\ \mbox{if} \ \pi\in M^+ \\
\\
+\infty, \ \mbox{otherwise}.
\end{array}\right.
\end{eqnarray*}
In the above computation we use that if $\pi\notin M^{+}(X\times Y \times Z \times W)$, there exists $u\leq 0$ such that $\langle\pi,u\rangle >0$. From the hypothesis of that $-c\geq 0$ we get $u \leq -c +0 +0$ where $b=0$ and $d=0$ have zero pressure (they are normalized). The same argument can be used for $\lambda u, \, \lambda \to +\infty$, showing that $\Theta^{*}(-\pi)=+\infty$.

Analogously
$$\begin{array}{l}
\Xi^*(\pi)=\displaystyle{\sup_{u\in E}}\left\{\langle \pi,u \rangle - \Xi(u)\right\} \\ \\
=\displaystyle{\sup_{ (\varphi,\psi,f,g)}\left\{ \langle \pi, \varphi(x)+\psi(y)- g(\tau_x(z))+ g (z) -f(\tau_y(w))+f(w) \rangle- \int_X \varphi\, d\mu -\int \psi\,d\nu  \right\}.}\\ \\
=\left\{\begin{array}{l}
 \displaystyle{0,  \, \mbox{if} \,  \pi \,\text{satisfies}\, (\ref{biholonomic}),(\ref{biholonomic2}) }
\\ \\
+\infty, \, \mbox{otherwise}.
\end{array}\right.
\end{array}$$

We observe that if $\Theta^{*}(-\pi)<+\infty$ and $\Xi^{*}(\pi)<+\infty$ then $\pi \in   \Pi(\mu,\nu,\tau)$. In this case we get  $$\displaystyle -\sup_{b\,: \,  P_\alpha(b)=0} \pi( b)=H_{\alpha}(\pi)$$ and $$\displaystyle -\sup_{d\,: \,  P_\beta(d)=0} \pi( d)=H_{\beta}(\pi).$$

Let
$\Phi_c$ be the set of continuous functions $\varphi(x)$ and $\psi(y)$ such that
$$ c(x,y,z,w) -\varphi(x)-\psi(z) +g(\tau_x(z))-g(z) + f(\tau_y(w)) -f(w) \leq b(x,z)+d(y,w)$$
for some continuous functions $f,g,b,d$ with $P_{\alpha}(b)=P_{\beta}(d) = 0$.

From (\ref{rockafeller}) we get
\[\inf_{(\varphi,\psi)\in\Phi_c} \int_X \varphi \, d\mu + \int_Y \psi\,d\nu =  \inf_{u\in E}[\Theta(u)+\Xi(u)]=\]
\begin{eqnarray*}
&&=\sup_{\pi\in E^{*}}[-\Theta^{*}(-\pi)-\Xi^{*}(\pi)]
=\sup_{\pi \in E^{*}}\left\{
\begin{array}{ll}
\displaystyle \pi(c) +H_\alpha(\pi)+H_{\beta}(\pi)\, \,   ,
 & \mbox{if} \ \pi\in\Pi(\mu,\nu,\tau) \\
\\
-\infty,  &  \mbox{otherwise}
\end{array}
\right\}\\
&&= \sup_{\pi\in \Pi(\mu,\nu,\tau)} \{\pi(c)+H_{\alpha}(\pi)+H_{\beta}(\pi) \}= \sup_{\pi \in \Pi(\mu,\nu,\tau)} \int  c  \, d\pi\, +\, H_\alpha(\pi)+H_\beta(\pi).
\end{eqnarray*}

Now, we are going to show that
\[\inf_{P_{\alpha,\beta}(c-\varphi-\psi)=0} \int_X \varphi \, d\mu + \int_Y \psi\,d\nu = \inf_{(\varphi,\psi)\in\Phi_c} \int_X \varphi \, d\mu + \int_Y \psi\,d\nu  \]
\[=\sup_{\pi \in \Pi(\mu,\nu,\tau)} \int  c  \, d\pi\, +\, H_\alpha(\pi)+H_\beta(\pi).\]
The second equality was proved above.
If $(\varphi,\psi)\in\Phi_c$ then there exist $f$, $g$, $b$, $d$ such that
\[c(x,y,z,w) -\varphi(x)-\psi(y) +g(\tau_x(z))-g(z) + f(\tau_y(w)) -f(w) \leq b(x,z)+d(y,w)\]
and for any $\pi \in \Pi(\cdot,\cdot,\tau)$ we have
\[\int  c(x,y,z,w) -\varphi(x)-\psi(y) \,d\pi +H_{\alpha}(\pi)+H_\beta(\pi)  \leq P_\alpha(b) + P_\beta(d) = 0.\]
Therefore $P_{\alpha,\beta}(c-\varphi-\psi) \leq 0$. This shows that
\[\inf_{P_{\alpha,\beta}(c-\varphi-\psi)\leq 0} \int_X \varphi \, d\mu + \int_Y \psi\,d\nu \leq \inf_{(\varphi,\psi)\in\Phi_c} \int_X \varphi \, d\mu + \int_Y \psi\,d\nu .\]
If $P_{\alpha,\beta}(c-\varphi-\psi)< 0$ there exists a number $a>0$ such that
$P_{\alpha,\beta}(c-\varphi-\psi+a)= 0$. If we denote $\hat \psi = \psi-a$ we get
$P_{\alpha,\beta}(c-\varphi-\hat \psi)= 0$ and $$\int_X \varphi \, d\mu + \int_Y \hat \psi\,d\nu = \int_X \varphi \, d\mu + \int_Y \psi\,d\nu -a < \int_X \varphi \, d\mu + \int_Y  \psi\,d\nu.$$
This shows that
\[ \inf_{P_{\alpha,\beta}(c-\varphi-\psi)\leq 0} \int_X \varphi \, d\mu + \int_Y \psi\,d\nu =\inf_{P_{\alpha,\beta}(c-\varphi-\psi)= 0} \int_X \varphi \, d\mu + \int_Y \psi\,d\nu.\]
Thus, we conclude that
\[\inf_{P_{\alpha,\beta}(c-\varphi-\psi)= 0} \int_X \varphi \, d\mu + \int_Y \psi\,d\nu \leq \sup_{\pi \in \Pi(\mu,\nu,\tau)} \int  c  \, d\pi\, +\, H_\alpha(\pi)+H_\beta(\pi).\]
On the other hand, if  $P_{\alpha,\beta}(c-\varphi-\psi)= 0$ and $\pi\in \Pi(\mu,\nu,\tau)$ then
\[\int c\, d\pi -\int \varphi\, d\mu -\int \psi\, d\nu + H_{\alpha}(\pi) +H_{\beta}(\pi) \leq 0\]
what means that
\[\int \varphi\, d\mu +\int \psi\, d\nu \geq  \int c\, d\pi  + H_{\alpha}(\pi) +H_{\beta}(\pi).\]
Finally, we conclude that
\[ \inf_{P_{\alpha,\beta}(c-\varphi-\psi)= 0} \int_X \varphi \, d\mu + \int_Y \psi\,d\nu \geq \sup_{\pi \in \Pi(\mu,\nu,\tau)} \int  c  \, d\pi\, +\, H_\alpha(\pi)+H_\beta(\pi).\]
\end{proof}

\noindent
\begin{remark} This proof can be extended to a general case where for each $i=1,2,...,n$ we have a contractible IFS $\{\tau_x:Z_i\to Z_i,\, x\in X_i\}$.
In this case, following analogous definitions, the above theorem can be stated as
\[\inf_{[P_{\alpha_1,...,\alpha_n}(c-(\varphi_1+...+\varphi_n))= 0]} \sum_{i=1}^n \int_{X_i} \varphi_i \, d\mu_i = \sup_{\pi \in \Pi(\mu_1,...,\mu_n,\tau)} \int  c  \, d\pi\, +\sum_{i=1}^{n}\, H_{\alpha_i}(\pi).\]
\end{remark}
\bigskip

\noindent
\textbf{ Proof of Proposition \ref{varpri2}.} \newline
It follows similar arguments defining  $\Xi: E\longrightarrow \mathbb{R}\cup\{+\infty \}$ from
$$
\Xi(u)=\left\{\begin{array}{ll}
\lambda &  \mbox{if} \ u=\lambda-g(\tau_x(z))+g(z)-f(\tau_y(w))+f(w), \\
&\mbox{where} \ \text{the functions are continuous}\,,\\
\\
+\infty,&  \mbox{otherwise}.
\end{array}\right.
$$

In this case

$$
\Xi^*(\pi)=\left\{\begin{array}{l}
 \displaystyle{0,  \, \mbox{if} \,  \pi \,\text{satisfy}\, (\ref{biholonomic}) }
\\ \\
+\infty, \, \mbox{otherwise}.
\end{array}\right.
$$

If $\Theta^{*}(-\pi)<+\infty$ and $\Xi^{*}(\pi)<+\infty$ then $\pi \in   \Pi(\cdot, \cdot,\tau)$. Let
$\Phi$ be the set of numbers $\lambda$ such that
$$ c(x,y,z,w) -\lambda -g(\tau_x(z))+g(z) - f(\tau_y(w)) +f(w) \leq b(x,z)+d(y,w)$$
for some continuous functions $f,g,b,d$ with $P_{\alpha}(b)=P_{\beta}(d) = 0$.

From (\ref{rockafeller}) we get
\[\inf_{\lambda \in\Phi} \lambda
= \sup_{\pi \in \Pi(\cdot,\cdot,\tau)} \int  c  \, d\pi\, +\, H_\alpha(\pi)+H_\beta(\pi).
\]
In order to finish the proof note that in the inequality
$$c(x,y,z,w) -\lambda -g(\tau_x(z))+g(z) - f(\tau_y(w)) +f(w) \leq b(x,z)+d(y,w)$$
we have
$P_\alpha(b(x,z)+g(\tau_x(z))-g(z))=0$ and $P_\beta(d(y,w)+f(\tau_y(w))-f(w))=0$.
\qed

\bigskip

The next result is related with the zero temperature case in Spin Lattice Systems (when the temperature is dropped) and with questions in ergodic optimization (see \cite{GL}). This result corresponds to the Kantorovich Duality for compact spaces and continuous cost function $-c$ if $Z$ and $W$ have only one element.

\begin{theorem}\label{kantorovich}
Let $\Phi_c$ be the set of continuous functions $\varphi(x)$, $\psi(y)$ satisfying
$$ c(x,y,z,w)  +g(\tau_x(z))-g(z) + f(\tau_y(w)) -f(w) \leq \varphi(x)+\psi(z)$$
for some functions $f\in C(W) $ and $g\in C(Z)$. Then
\[ \inf_{(\varphi,\psi) \in \Phi_c} \int \varphi(x)\,d\mu +\int \psi(y)\,d\nu = \sup_{\pi\in\Pi(\mu,\nu,\tau)} \int c\, d\pi.\]
\end{theorem}
\begin{proof}
The proof follows the same reasoning presented in the previous theorem, choosing
$$\Theta(u)=\left\{\begin{array}{ll}
0,&  \mbox{if} \ u(x,z,y,w)\geq c(x,y,z,w), \; \forall(x,y,z,w)\\
\\

+\infty,&  \mbox{otherwise}
\end{array}\right.
$$
and
$$
\Xi(u)=\left\{\begin{array}{ll}
\int_X\varphi \,d\mu + \int_Y \psi\, d\nu,&  \mbox{if} \ u=\varphi(x)+\psi(y)-g(\tau_x(z))+g(z)-f(\tau_y(w))+f(w), \\
&\mbox{where} \ \text{the functions are continuous}\,,\\
\\
+\infty,&  \mbox{otherwise}.
\end{array}\right.
$$
\end{proof}

If we suppose in Theorem~\ref{varprindual} that $Y=\{y_0\}$ and $W=\{w_0\}$ we obtain the next result.

\begin{corollary}
For a fixed $\mu \in \mathcal{P}(X)$ and $c(x,z)$ continuous  we have
\begin{equation}\label{dual}
\inf_{P_\alpha(c-\varphi(x))=0}\int \varphi\, d\mu  = \sup_{\pi\in\Pi(\mu,\tau)} \int c\, d\pi + H_\alpha(\pi).
\end{equation}
\end{corollary}

In some sense the concept of eigenvalue was changed in this result. For the purpose of the above result we can try to think that the eigenvalue is a function on the $x$ variable.
The equation $L(h)=\lambda h$ should be changed for the existence of functions $h(z)$ and $\varphi(x)$ such that $L(h)=\varphi \cdot h$. But the left hand side is a function of the variable $z$ and the right hand side is a function (or product of functions) of the variables $x$ and $z$. We return to the original equation $L(h)=\lambda h$ and rewrite this in the form $L(\frac{h}{\lambda})=h$. In this way we can try to find functions $\varphi(x)$ and $h(z)$ such that
\[\int e^{c(x,z)-\varphi(x)}h(\tau_x(z)) = h(z).\]
We observe  that there exist too many pairs of solutions. Indeed, for each fixed $\varphi(x)$ we can apply the Lemma~\ref{exist} for $L_{c-\varphi}$ determining $\lambda>0$ and $h>0$ satisfying $L_{c-\varphi}h=\lambda\cdot h$. This can be rewritten in the form
 \[\int e^{c(x,z)-\varphi(x)-\log(\lambda)}h(\tau_x(z)) = h(z).\]
Hence, there is a function $\hat \varphi(x) = \varphi(x)+\log(\lambda)$ and a function $h>0$ such that
 \[\int e^{c(x,z)-\hat \varphi(x)}h(\tau_x(z)) = h(z).\]
Thus any function $\varphi(x)$ plays the role of an ``eigenvalue'' except by the addition of a constant. For a given cost function $c$ there exist too many ways of get a normalization by adding a function $\varphi(x)$ and a function in the form $g(z)-g(\tau_x(z))$. In the previous section we make the normalization adding a constant and not a function $\varphi(x)$.

The next result can be interpreted as a kind of slackness condition in the present setting.

\begin{proposition} Let $c(x,z)$ and $\varphi(x)$ be Lipschitz functions and $\pi \in \Pi(\mu,\tau)$. If $P(c-\varphi)=0$ and $\pi$ is the holonomic probability associated to $c-\varphi$, then $\varphi$ and $\pi$ realize the infimum and the supremum in (\ref{dual}).
\end{proposition}

\begin{proof} We know that
\[0 = P(c-\varphi) = \int c-\varphi\, d\pi +H(\pi).\]
Then
\[\int \varphi\, d\mu =\int c\, d\pi +H(\pi).\]
This is possible only if $\varphi$ realizes the infimum and $\pi$ realizes the supremum in (\ref{dual}).
\end{proof}

\vspace{0.9cm}

\begin{thebibliography}{999}



\bibitem{BCLMS}
A. T. Baraviera, L. Cioletti, A. O.  Lopes, J. Mohr, R. R. Souza,
``On the general one-dimensional XY model: positive and zero temperature, selection and non-selection''. \emph{Rev. Math. Phys.} 23, no. 10, 1063--1113 (2011).

\bibitem{BarnDemElton} M. F. Barnsley, S. G. Demko, J. H. Elton, J. S. Geronimo,   ``Invariant measures for
Markov processes arising from iterated function systems with
place-dependent probabilities.'' \emph{Ann. Inst. H. Poincar\'e
Probab. Statist.} 24, no. 3, 367--394 (1988).

\bibitem{PBern} P. Bernard,   ``Young measures, superposition and transport." \emph{Indiana Univ. Math. J.} 57, no. 1, 247--275 (2008).

\bibitem{BG} A. Biryuk,  D. Gomes,   ``An introduction to the Aubry-Mather theory." \emph{S\~ao Paulo J. Math. Sci.} 4, no. 1, 17--63 (2010).

\bibitem{bousch} T. Bousch,  ``La condition de Walters''. \emph{Ann. Sci. ENS}. 34, issue 2, 287--311 (2001).

\bibitem{VE} V. Enter, R. Fernandez, A. Sokal, ``Regularity properties and pathologies of position-space renormalization-group transformations: Scope and limitations of Gibbsian theory''. \emph{Journal of Statistical Physics}. 72, issue 5-6,  879--1167 (1993).

\bibitem{GL}  E. Garibaldi, A.  Lopes, ``On the Aubry–Mather theory for symbolic dynamics''. \emph{Erg. Theo. and Dyn. Syst}. 28,  791--815 (2008)

\bibitem{GO} D. Gomes, E. Oliveira,  ``Mather problem and viscosity solutions in the stationary setting''. \emph{S\~ao Paulo J. Math. Sci.} 6, no. 2, 301--334 (2012)

\bibitem{BJ} P. Jorgensen,  Analysis and Probability:
wavelets, signals, fractals, Springer-Verlag (2006).


\bibitem{LM}
A. Lopes, J. Mengue, ``Duality Theorems in Ergodic Transport'', \emph{Journal of Statistical Physics}. 149, issue 5,  921-942 (2012).

\bibitem{LMMS}
A. Lopes, J. Mengue, J. Mohr, R. Souza, ``Entropy and Variational Principle for one-dimensional Lattice Systems with a general a-priori probability: positive and zero temperature" to appear in \emph{Erg. Theo. and Dyn. Syst}.

\bibitem{LMMS2}
A. Lopes, J. Mengue, J. Mohr and R. Souza, ``Entropy, Pressure and Duality for Gibbs plans in Ergodic transport'' (arXiv:1308.6514v3 [math.DS])

\bibitem{OlivLop} A.  Lopes, E. Oliveira,   ``Entropy and variational principles for holonomic probabilities of IFS''. \emph{Discrete Contin. Dyn. Syst.} 23, no. 3, 937--955 (2009).

\bibitem{P} P. Walters, An introduction to ergodic theory, Springer-Verlag (1982)

\bibitem{PP} W. Parry, M.  Pollicott, Zeta functions and the periodic orbit structure of hyperbolic dynamics, \emph{Ast\'erisque} Vol {187-188} (1990).

\bibitem{STE} O. Stenflo, ``Uniqueness of invariant measures for place-dependent random iterations of functions." Fractals in multimedia (Minneapolis, MN, 2001),  IMA Vol. Math. Appl. 132, Springer,  13--32 (2002).


\bibitem{Urba} M. Urbanski,  ``Hausdorff measures versus
equilibrium states of conformal infinite iterated function
systems.'', International Conference on Dimension and Dynamics
(Miskolc, 1998). Period. Math. Hungar. 37, no. 1-3,
153--205 (1998).

\bibitem{Vi1}
C. Villani, Topics in optimal transportation, AMS, Providence (2003).





\end{thebibliography}
\end{document}